\documentclass[11pt,a4paper]{article}
\usepackage{amsthm,amssymb,amsfonts,amsmath,lineno,mathtools,graphicx}

\usepackage{cite} 

\usepackage{enumerate}

\newcommand\blfootnote[1]{%
 \begingroup
 \renewcommand\thefootnote{}\footnote{#1}%
 \addtocounter{footnote}{-1}%
 \endgroup
}

\newcommand{\smodul}[1]{\,\, (\textrm{mod }#1) }

\newtheorem{theorem}{Theorem}

\newtheorem{corollary}{Corollary}

\date{}

\begin{document}

\title{Trees whose even-degree vertices induce a path are antimagic}

\author{Antoni Lozano\thanks{Computer Science Department, Universitat Polit\`ecnica de Catalunya, Spain, {\tt antoni@cs.upc.edu}.} \and
Merc\`e Mora\thanks{Mathematics Department, Universitat Polit\`ecnica de Catalunya, Spain, {\tt merce.mora@upc.edu}.}
\and Carlos Seara\thanks{Mathematics Department, Universitat Polit\`ecnica de Catalunya, Spain, {\tt carlos.seara@upc.edu}.}
\and Joaqu\'in Tey\thanks{Math. Department, Universidad Aut\'onoma Metropolitana-Iztapalapa, M\'exico, {\tt jtey@xanum.uam.mx}}
}

\maketitle

\blfootnote{
	\begin{minipage}[l]{0.3\textwidth} \includegraphics[trim=10cm 6cm 10cm 5cm,clip,scale=0.15]{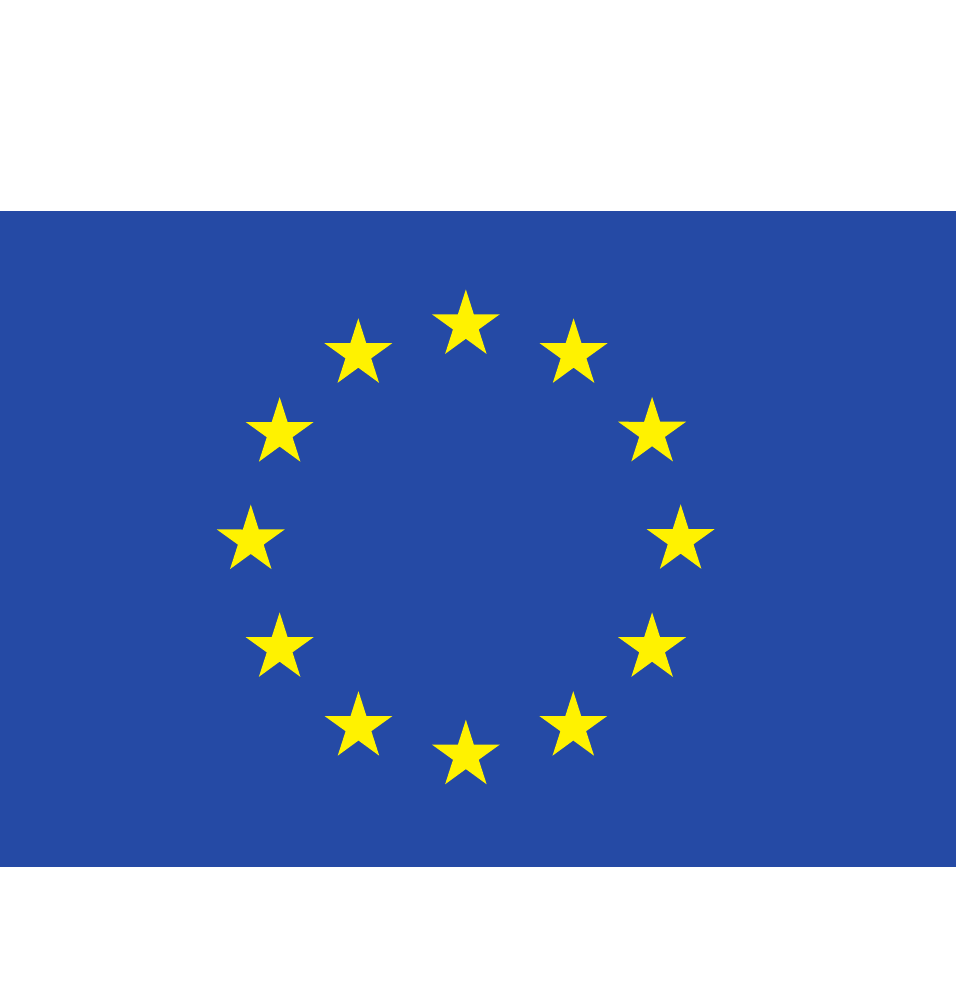} \end{minipage} \hspace{-2.5cm} \begin{minipage}[l][1cm]{0.8\textwidth}
		This project has received funding from the European Union's Horizon 2020 research and innovation programme under the Marie Sk\l{}odowska-Curie grant agreement No 734922.
	\end{minipage}
}

\begin{abstract}\noindent
An \emph{antimagic labeling} of a connected graph $G$ is a bijection from the set of edges $E(G)$ to $\{1,2,\dots,|E(G)|\}$ such that all vertex sums are pairwise distinct, where the \emph{vertex sum} at vertex $v$ is the sum of the labels assigned to edges incident to $v$. A graph is called \emph{antimagic} if it has an antimagic labeling. In 1990, Hartsfield and Ringel conjectured that every simple connected graph other than $K_2$ is antimagic; however the conjecture remains open, even for trees. In this note we prove that trees whose vertices of even degree induce a path are antimagic, extending a result given by Liang, Wong, and Zhu [\emph{Discrete Math.} 331 (2014) 9--14].
\end{abstract}

\section{Introduction}\label{sec:int}

All graphs considered in this work are finite, undirected and simple. Given a graph $G=(V(G),E(G))$ and a vertex $v\in V(G)$, we denote by $E_G(v)$ the set of edges incident to $v$ and by $d_G(v)=|E_G(v)|$, the degree of $v$ in $G$. We will just write $E(v)$ and $d(v)$ when $G$ is clear from context. A tree is a connected and acyclic graph, and a forest is a disjoint union of trees. Undefined terms in this work can be found in~\cite{ChLZ}.

An \emph{(edge) labeling} of a graph $G$ is a mapping from $E(G)$ to the set of nonnegative integers. A labeling $\phi$ of a connected graph $G$ is called \emph{antimagic} if it is a bijection $\phi : E(G)\rightarrow \{1,2,\dots,|E(G)|\}$ such that all vertex sums are pairwise distinct, where the \emph{vertex sum} $s(v)$ at vertex $v\in V(G)$ is defined as $s(v)=\sum_{e\in E(v)}\phi(e)$. A graph is called \emph{antimagic} if it has an antimagic labeling. 

In 1990, Hartsfield and Ringel~\cite{HR} conjectured that every simple connected graph other than $K_2$ is antimagic. The conjecture has received much attention (see~\cite{G}); but it is widely open in general, even for trees. Nevertheless, several classes of trees are known to be antimagic (see~\cite{CCLP,DL,HR, KLR,LWZ,LMS,LMST,Sh}).

Given a tree $T$, $V_{even}(T)$ (resp. $V_{odd}(T)$) denotes the set of even (resp. odd) degree vertices of $T$. Regarding trees such that $V_{even}$ induces a path, Liang, Wong, and Zhu~\cite{LWZ} proved the following two theorems.

\begin{theorem}{\rm \cite{LWZ}}\label{Even degree odd}
If $T$ is a tree such that $V_{even}(T)$ induces a path and $\vert V_{even}(T)\vert$ is odd, then $T$ is antimagic.
\end{theorem}

\begin{theorem}{\rm \cite{LWZ}}\label{Even degree even}
Let $T$ be a tree such that $V_{even}(T)$ induces a path of length $2p$, $(v_1,\dots ,v_{2p})$. Let  $v_0$ (resp. $v_{2p+1}$) be a neighbor of $v_1$ (resp. $v_{2p}$) different from $v_2$ (resp. $v_{2p-1}$). If $d(v_p)\neq d(v_{2p+1})+1$ or $d(v_{p+1})\neq d(v_0)+1$, then $T$ is antimagic.
\end{theorem}

The aim of this note is to extend Theorem~\ref{Even degree even} to all cases, that is, to prove the antimagicness of trees such that $V_{even}(T)$ induces a path whenever $\vert V_{even}(T)\vert$ is even, obtaining as a consequence that trees whose even-degree vertices induce a path are antimagic.

\section{Constructing an antimagic labeling}\label{sec:result}

In the proof of the next theorem we follow and extend the main idea developed by Liang, Wong, and Zhu in~\cite{LWZ}.
We denote by $[a,b]$ the set of consecutive integers $\{a,a+1,\dots,b\}$, where $a\leq b$.

\begin{theorem}\label{V even path}
If $T$ is a tree such that $V_{even}(T)$ induces a path and $|V_{even}(T)|$ is even, then $T$ is antimagic.
\end{theorem}

\begin{proof}

It is known that trees without vertices of degree $2$ are antimagic~\cite{KLR,LWZ}, hence we may assume $|V_{even}(T)|=2p$, with $p\geq 1$. Let $P=(v_0,v_1,v_2,\dots,v_{2p},v_{2p+1})$ be a path induced by $V_{even}(T)$ extended with endpoints in $V_{odd}(T)$, that is, $V_{even}(T)=\{v_1,\dots,v_{2p}\}$ and $\{v_0,v_{2p+1}\}\subseteq V_{odd}(T)$. For every $1\le i\le 2p+1$, we denote by $e_{i}$ the edge $v_{i-1}v_i$.

We prove the theorem by constructing an antimagic labeling $\phi$ of $T$ in two steps. The first step produces a labeling of a subtree of $T$ containing the path $P$ and satisfying a particular additional condition. This labeling will be extended to an antimagic labeling of $T$ at the second step.

Let $m=|E(T)|$. We will use the residues modulo $m+2$ to compare vertex sums: since vertex sums are distinct if they are distinct  modulo $m+2$, it is enough to compare vertex sums whenever they are equal modulo $m+2$ in order to check that they all are pairwise distinct.

For each (not necessarily connected) subgraph $T'$ of $T$, we set
$L_{\phi}(T')=\{\phi(e) \,:\, e\in E(T')\}$ and
$s_{T'}(v)=\sum_{e\in E_{T'}(v)} \phi(e)$ for every $v \in V(T')$ such that $d_{T'}(v)\geq 1$.
Obviously, if $T'=T$, then $s_{T'}=s$.
The set of all vertex sums modulo $m+2$ in $T'$ will be denoted by $R_{m+2}(T')$, that is,
$$R_{m+2}(T')=\{s_{T'}(v)\smodul {m+2}  \,:\, v\in V(T') \text{ and } d_{T'}(v) \geq 1\}\subseteq\{0,1,\dots,m+1\}.$$
\medskip

\noindent {\bf STEP I.}
The labeling of the tree $T_1$ constructed at this step will satisfy the following condition: all vertex sums in $T_1$ will be pairwise distinct modulo $m+2$ with at most one exception; moreover, if the vertex sums are equal modulo $m+2$ for a pair of vertices, then exactly one of them will be a leaf in $T$ and the vertex sums in $T_1$ for both vertices in the pair will be different.

As a starting point, let $T_1:=P$ and define
\begin{center}
\boxed{
	$$
	\begin{array}{lr}
	\phi(e_{2i+1}) := i+1, &\hbox { for }0\le i\le p;\\
	\phi(e_{2i}) := m-p+i, &\hbox { for }1\le i\le p.
	\end{array}
	$$}
\end{center}

Hence,
\begin{equation}\label{Eq1}
L_{\phi}(T_1)=[1,p+1]\cup [m-p+1,m].
\end{equation}
Moreover, $s_{T_1}(v_0)=1$, $s_{T_1}(v_{2p+1})=p+1$, and for $1\leq i\leq 2p$, $s_{T_1}(v_i)=m-p+i+1$. Next, we calculate  the set $R_{m+2}(T_1)$ according to the values of $p$ and to the degrees of $v_0$ and $v_{2p+1}$.
\medskip

\noindent\emph{Case 1.} $p=1$. In this case, $|V(T_1)|=4$ and
\begin{equation}\label{Eq2}
R_{m+2}(T_1)=\{0,1,2,m+1\}.
\end{equation}
Hence, vertex sums at the vertices of $T_1$ are distinct.
\medskip

\noindent\emph{Case 2.} $p>1$. In such a case,
\begin{equation}\label{Eq3}
R_{m+2}(T_1)=\Big( [0,p+1]\setminus \{p\} \Big)\cup  [m-p+2,m+1],
\end{equation}
and, only the residues of vertex sums at vertices $v_0$ and $v_{p+2}$ are equal.
In fact, we have that $s_{T_1}(v_0)\equiv 1 \equiv s_{T_1}(v_{p+2}) \pmod {m+2}$ (see an example in Figures~\ref{fig_ant1}(a) and~\ref{fig_ant2}(a)).

\begin{figure}[ht!]
\centering
\includegraphics[width=\textwidth]{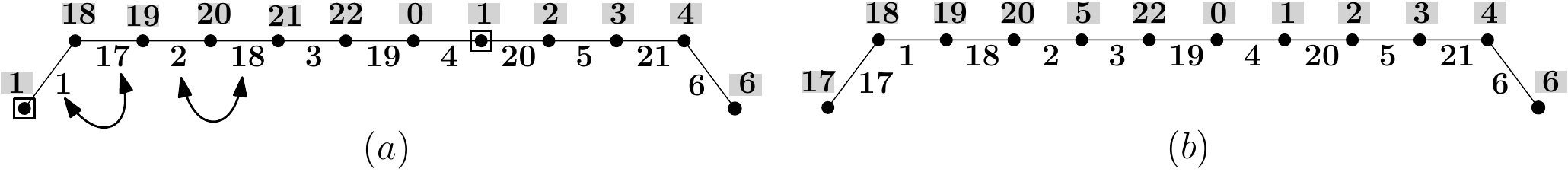}
\caption{Labeling of $T_1$ for $p=5$ and $m=21$; (a) before the swaps, and (b) after the swaps. The shadowed label at each vertex is the vertex sum modulo $23$. Squared vertices have the same label.}\label{fig_ant1}
\end{figure}

\begin{figure}[ht!]
\centering
\includegraphics[width=\textwidth]{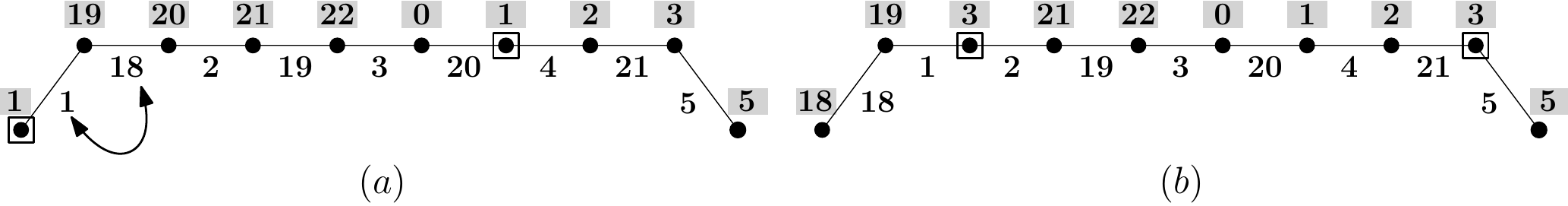}
\caption{Labeling of $T_1$ for $p=4$ and $m=21$; (a) before the swap, and (b) after the swap. The shadowed label at each vertex is the vertex sum modulo $23$. Squared vertices have the same label.}
\label{fig_ant2}
\end{figure}

Now we distinguish two subcases.

\begin{enumerate}[{\textrm 2}.1.]
\item
\emph{At least one of the vertices $v_0$ or $v_{2p+1}$ is a leaf in $T$.} Notice that by properly relabeling the vertices of $T_1$, we may assume $d_T(v_{0})=1$. Then,
 $s_{T_1}(v_0) \equiv 1 \equiv  s_{T_1}(v_{p+2})  \pmod {m+2}$, but
 $s_{T_1}(v_0)=1<m+3=s_{T_1}(v_{p+2})$.

\item
\emph{Neither $v_0$ nor $v_{2p+1}$ are leaves in $T$.} In this case; for $1\le i\le \lfloor \small{\frac{p-1}2} \rfloor$, we swap the labels of the edges $e_{2i-1}$ and $e_{2i}$, that is,

\begin{center}
\boxed{		$$	
		\begin{array}{lr}
		\phi(e_{2i-1}):=  m-p+i, &\hbox{ for }1\le i\le \lfloor \small{\frac{p-1}2} \rfloor;\\
		\phi(e_{2i}):= i,  &\hbox{ for } 1\le i\le \lfloor \frac{p-1}2 \rfloor.
		\end{array}
		$$
}\end{center}

Notice that the endpoints of the subpath of $T_1$ involved in the swaps are $v_0$ and $v_k$, where $k=p-2$, if $p$ is even; and $k=p-1$, if $p$ is odd.
After the swaps, $s_{T_1}(v_0)=m-p+1$; also it can be easily checked that $s_{T_1}(v_k)=p-1$, if $p$ is even; and $s_{T_1}(v_k)=p$, if $p$ is odd;
and the vertex sum at any other vertex in $T_1$ remains unchanged.
We distinguish cases depending on the parity of $p$.
\begin{enumerate}[(a)]
	\item \emph{$p$ odd}. In this case we have that $s_{T_1}(v_k)=p$, implying that the residues modulo $m+2$ at the vertices of $T_1$ are pairwise distinct. Concretely,
	\begin{equation}\label{Eq4}
	R_{m+2}(T_1)=[0,p+1]\cup \Big([m-p+1,m+1]\setminus \{m\} \Big).
	\end{equation}
	\item \emph{$p$ even}. Then, $s_{T_1}(v_k)=p-1\equiv s_{T_1}(v_{2p})\pmod {m+2}$, and thus
\begin{equation}\label{Eq5}
R_{m+2}(T_1)=\Big([0,p+1]\setminus \{ p \} \Big) \cup \Big([m-p+1,m+1]\setminus \{m-1\}\Big).
\end{equation}
Notice that only $s_{T_1}(v_k)$ and $s_{T_1}(v_{2p})$ have the same residue  in $T_1$
(see an example in Figure~\ref{fig_ant2}(b)).
Now, let $x_0=v_{2p+1}$ and let $P'=(x_0,x_1,\dots ,x_{\ell})$ be a maximal subpath of $T$ starting at $x_0$ and with vertices in $V_{odd}(T)$. Observe that, in such a case,
$x_{\ell}$ is a leaf in $T$ and there exist vertices $y_0,\dots ,y_{\ell-1}\in V_{odd}(T)$ such that $x_iy_i\in E(T)$ (see an example in Figure~\ref{fig_ant3}).
We update $T_1$ as the tree induced by the set of vertices of the paths $P$ and $P'$, and $\{ y_0,\dots ,y_{\ell-1}\}$:
$$ T_1 := T [ \{ v_0,\dots ,v_{2p}\}\cup \{ x_0,\dots ,x_{\ell}\} \cup \{ y_0,\dots ,y_{\ell-1}\} ].  $$

\begin{figure}[ht]
\centering
\includegraphics[width=0.7\textwidth]{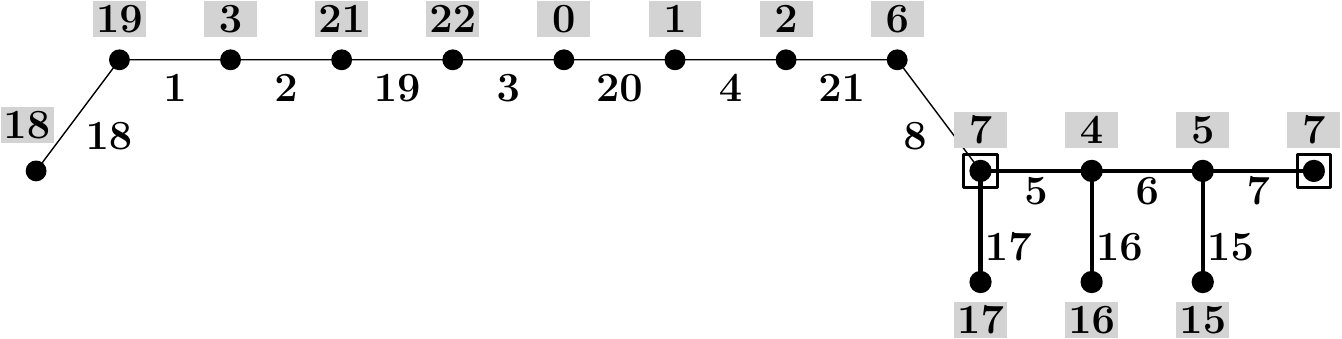}
\caption{Labeling of $T_1$ with $p=4$, $\ell=3$, and $m=21$. The shadowed label at each vertex is the vertex sum modulo $23$. Squared vertices have the same label.}\label{fig_ant3}
\end{figure}

We  define the labels of the new edges and update the label of the edge $e_{2p+1}$ as follows:
\begin{center}
	\boxed{
$$	\begin{array}{lr}
		\phi (e_{2p+1}):=p+\ell +1;&\\
		\phi (x_ix_{i+1}):=p+i+1, &\hbox{ for }0\le i\le \ell -1;\\
		\phi (x_iy_{i}):=m-p-i,&\hbox{ for }0\le i\le \ell -1.
	\end{array}$$}
\end{center}

Thus, we have
\begin{equation}\label{Eq6}
L_\phi(T_1)=[1,p+\ell+1]\cup [m-p-\ell+1,m],
\end{equation}
and

\begin{align*}
s_{T_1} &(x_i)=m+p+i+1\equiv {p+i-1}\pmod {m+2}, \text{ for } 1\leq i \leq \ell-1,\\
s_{T_1} &(v_{2p})=m+p+\ell+1\equiv {p+\ell-1}\pmod {m+2},\\
s_{T_1} &(x_0)=m+p+\ell+2\equiv {p+\ell}\pmod {m+2},\\
s_{T_1} &(x_\ell)=p+\ell\equiv {p+\ell}\pmod {m+2},\\
s_{T_1} &(y_i)=m-p-i\equiv {m-p-i}\pmod {m+2}, \text{ for } 0\leq i \leq \ell-1.
\end{align*}
Therefore, taking into account Equality \eqref{Eq5},
\begin{equation}\label{Eq7}
R_{m+2}(T_1)=[0,p+\ell]\cup \Big([m-p-\ell+1,m+1]\setminus \{m-1\}\Big)
\end{equation}
and only $s_{T_1}(x_0)$ and  $s_{T_1}(x_\ell)$ have the same residue, concretely $p+\ell$.
However, $s_{T_1}(x_0)=m+p+\ell+2 >p+\ell=s_{T_1}(x_\ell)$, and hence all vertex sums in $T_1$ are different (see an example in Figure~\ref{fig_ant3}).
\vspace{0.1cm}
\end{enumerate}
\end{enumerate}

Notice that, in each of the above cases, $|L_\phi(T_1)|=|E(T_1)|$. Hence $\phi$, restricted to $E(T_1)$, is a bijection from $E(T_1)$ to $L_\phi(T_1)$.
\medskip

\noindent
{\bf STEP II.} Now, let $T_2$ be the forest obtained by removing all the edges of $T_1$. Each component of $T_2$ has exactly one vertex in $T_1$. Therefore, if $T_2(v)$ denotes the component of $T_2$ containing $v$, $$T_2:=T-E(T_1)=\bigcup_{v\in V(T_1)} T_2(v).$$

Clearly, $T_{2}(v)$ can be viewed as a directed rooted tree with root at $v$, where every edge is directed away from the root. Moreover, since every vertex of  $T_2(v)$ different from $v$ has odd degree in $T$, each vertex  in $T_2(v)$ has an even number of children in this rooted tree and, therefore, $|E(T_2(v))|$ is even for every $v\in V(T_1)$. Hence, $|E(T_2)|$ is even.
If we set $\ell =0$ whenever $T_1=P$, then, by Equalities~\eqref{Eq1} and~\eqref{Eq6}, the available labels for the edges of $T_2$ are
\begin{equation}\label{Eq8}
L_\phi(T_2) :=
[p+\ell+2,m-p-\ell].
\end{equation}
As $L_\phi(T_2)=[a,b]$, where $a+b=m+2$, and each $w\in V(T_2)$ has an even number of children, we can label the edges of $T_2$ with integers in $L_\phi(T_2)$ fulfilling this additional condition: if a vertex $w$ has an outgoing edge with label $t$ in the corresponding rooted tree of $T_2$, then $w$ has another outgoing edge with label $m+2-t$ (see an example in Figure~\ref{fig_ant4}).
Concretely, if $|E(T_2)|=2r$  for some integer $r$, then we set $E(T_2)=E_1\cup \dots \cup E_r$, where $( E_1,\dots ,E_r)$ is a partition of $E(T_2)$ such that $E_i=\{f_{i}^1,f_{i}^2\}$ contains exactly two outgoing edges from the same vertex in some rooted tree of the forest $T_2$, and we label the edges of $T_2$ in the following way:

\begin{center}
\boxed{$$ \begin{array}{lr}
	\phi(f_{i}^1):=p+\ell +2 + (i-1),&\hbox{ for }1\le i\le r;\\
	\phi(f_{i}^2):= m-p-\ell -(i-1),&\hbox{ for }1\le i\le r.
	\end{array}$$}
\end{center}

\begin{figure}[t]
\centering
 \includegraphics[width=0.75\textwidth]{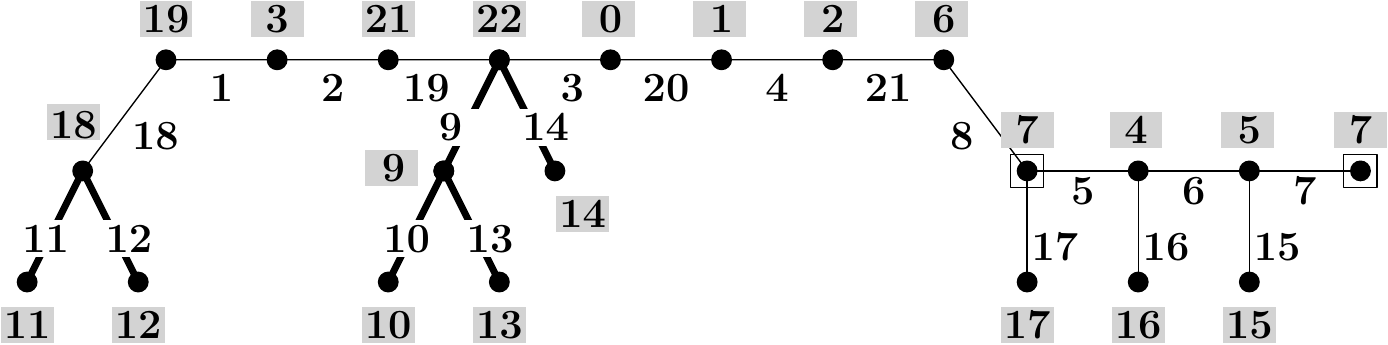}
\caption{An antimagic labeling of a tree with $m=21$. Thicker edges correspond to the forest $T_2$ and are labeled in Step~II (in this example, the forest $T_2$ has two nontrivial components). The shadowed label at each vertex is the vertex sum modulo $23$. Squared vertices have the same label, but different vertex sums.}\label{fig_ant4}
\end{figure}

\smallskip
Clearly, by the previous discussion, the labeling $\phi$ already constructed is a bijection from $E(T)$ to $[1,m]$. Finally, we just need to show that the vertex sums defined by $\phi$ in $T$ are pairwise distinct.

Observe that in $T_2$, the sum of the labels of the outgoing edges of $v$ is a multiple of $m+2$. Thus, the following two conditions hold.

\begin{enumerate}[1.]

\item\label{con1} For every $v\in V(T_1)$, $s(v)\equiv s_{T_1}(v)\pmod {m+2}$.

\item\label{con2} For every $v\in V(T)\setminus V(T_1)$, $s(v)\equiv \phi(f)\pmod {m+2}$, where $f$ is the incoming edge of $v$ in $T_2$.
\end{enumerate}

 Let $u,v\in V(T)$. We consider the following cases:
\begin{enumerate}[(a)]

\item $u,v\in V(T)\setminus V(T_1)$.
Since $\phi: E(T)\rightarrow [1,m]$ is a bijection, Condition~\ref{con2} implies that $s(u)\neq s(v)$.

\item $u,v\in V(T_1)$.
If $s_{T_1}(u)\not\equiv s_{T_1}(v)\pmod {m+2}$, then by Condition~\ref{con1} we have that $s(u)\neq s(v)$. Otherwise, as we have seen in Cases 2.1 and 2.2(b) of Step~I, one of these vertices, say $u$, is a leaf in $T$ and $s_{T_1}(u)<s_{T_1}(v)$. Therefore, we have that $s(u)=s_{T_1}(u)< s_{T_1}(v)\leq s(v)$, as we wanted to prove.

\item One of the vertices belongs to $V(T_1)$ and the other to $V(T)\setminus V(T_1)$. We can assume without loss of generality that $u\in V(T_1)$ and $v\in V(T)\setminus V(T_1)$. By Condition~\ref{con1}, $s(u)\pmod {m+2}\in R_{m+2}(T_1)$. Moreover, by Condition ~\ref{con2}, $s(v)\pmod {m+2}\in L_\phi(T_2)$.
By Equalities~\eqref{Eq2},~\eqref{Eq3},~\eqref{Eq4},~\eqref{Eq7}, and~\eqref{Eq8}, we have $R_{m+2}(T_1)\cap L_\phi(T_2) =\emptyset$. Hence, $s(u)\neq s(v)$.
\end{enumerate}
Thus, the theorem holds.
\end{proof}

The next result follows from Theorems~\ref{Even degree odd} and~\ref{V even path}.

\begin{corollary}
If $T$ is a tree such that $V_{even}(T)$ induces a path, then $T$ is antimagic.
\end{corollary}

\noindent{\bf Acknowledgments.}
A. Lozano is supported by the European Research Council (ERC) under the European Union's Horizon 2020 research and innovation programme (grant agreement ERC-2014-CoG 648276 AUTAR); M. Mora is supported by projects Gen. Cat. DGR 2017SGR1336, MINECO MTM2015-63791-R, and H2020-MSCA-RISE project 734922-CONNECT; and C. Seara is supported by projects Gen. Cat. DGR 2017SGR1640, MINECO MTM2015-63791-R, and H2020-MSCA-RISE project 734922-CONNECT.

\end{document}